\documentclass[proceedings]{aofa}
\usepackage{amsmath,amssymb}

\newtheorem{theorem}{Theorem}
\newtheorem{lemma}[theorem]{Lemma}

\newcommand{\infec} {A}

\newcommand{\checked} {Z}

\newcommand{\filter}[1]{\mathcal{F}(#1)}

\author{Mihyun ~Kang\thanks{Supported by Austrian Science Fund (FWF): P26826, W1230.} \and Tam\'as ~Makai \thanks{Supported by Austrian Science Fund (FWF): P26826.}
}
\title{Bootstrap percolation on G(n,p) revisited}

\address{
Graz University of Technology, Institute of Discrete Mathematics, Steyrergasse 30, 8010 Graz, Austria\\
\{kang, makai\}@math.tugraz.at
}

\keywords{Random graph, Bootstrap percolation, Martingale}

\begin{document}
\maketitle

\begin{abstract}
Bootstrap percolation on a graph with infection threshold \begin{math}r\in \naturals\end{math} is an infection process, which starts from a set of initially infected vertices and in each step every vertex with at least \begin{math}r\end{math} infected neighbours becomes infected. We consider bootstrap percolation on the binomial random graph \begin{math}G(n,p)\end{math}, which was investigated  among others by Janson, \L uczak, Turova and Valier (2012). We improve their results by strengthening the probability bounds for the number of infected vertices at the end of the process.
\end{abstract}

\section{Introduction}

Bootstrap percolation on a graph with infection threshold \begin{math}r\in \naturals\end{math} is a deterministic infection process which evolves in rounds. In each round every vertex has exactly one of two possible states: it is either infected or uninfected. We denote the set of initially infected vertices by \begin{math}\infec(0)\end{math}. In each round of the process every uninfected vertex \begin{math}v\end{math} becomes infected if it has at least \begin{math}r\end{math} infected neighbours, otherwise it remains uninfected. Once a vertex has become infected, it remains infected forever. The final infected set is denoted by \begin{math}\infec_f\end{math}.

Bootstrap percolation was introduced by Chalupa, Leath, and Reich \cite{bootstrapintr} in the context of
magnetic disordered systems. Since then bootstrap percolation processes (and extensions) have been used to describe several complex phenomena: from neuronal activity \cite{MR2728841,inhbootstrap} to the dynamics of the Ising model at zero temperature \cite{Fontes02stretchedexponential}.

In the context of social networks, bootstrap percolation provides a prototype model for the spread of ideas. In this setting infected vertices represent individuals who have already adopted a new belief and a person adopts a new belief if at least \begin{math}r\end{math} of his acquaintances have already adopted it.

On the \begin{math}d\end{math}-dimensional grid \begin{math}[n]^d\end{math} bootstrap percolation has been studied by 
Balogh, Bollob{\'a}s, Duminil-Copin, and Morris \cite{MR2888224}, when the initial infected set contains every vertex independently with probability \begin{math}p\end{math}. 
For the size of the final infection set they showed the existence of a sharp threshold. More precisely, they established the threshold probability \begin{math}p_\mathrm{c}\end{math}, such that if \begin{math}p\leq (1-\varepsilon )p_\mathrm{c}\end{math}, then the probability that every vertex in \begin{math}[n]^d\end{math} becomes infected tends to 0, as \begin{math}n\rightarrow\infty\end{math}, while if \begin{math}p\geq (1+\varepsilon )p_\mathrm{c}\end{math}, then the probability that every vertex in \begin{math}[n]^d\end{math} becomes infected tends to one, as \begin{math}n\rightarrow\infty\end{math}.

Bootstrap percolation has also been studied for several random graph models. For instance Amini and Fountoulakis \cite{bootpower} considered the Chung-Lu model \cite{MR1955514} where the vertex weights follow a power law degree distribution and the presence of an edge \begin{math}\{u,v\}\end{math} is proportional to the product of the weights of \begin{math}u\end{math} and \begin{math}v\end{math}. Taking into account that in this model a linear fraction of the vertices have degree less than \begin{math}r\end{math} and thus at most a linear fraction of the vertices can become infected, the authors proved the size of the final infected set \begin{math}\infec_f\end{math} exhibits a phase transition. 

Janson,  \L uczak, Turova, and Vallier \cite{MR3025687} analysed bootstrap percolation on the binomial random graph \begin{math}G(n,p)\end{math}, a graph with vertex set \begin{math}[n]:=\{1,2,\ldots, n\}\end{math} where every edge appears independently with probability \begin{math}p=p(n)\end{math}, and the set of initially infected vertices \begin{math}\infec(0)\end{math} is chosen uniformly at random from the vertex sets of size \begin{math}a\end{math}. For \begin{math}r\geq 2\end{math} and \begin{math}p\end{math} satisfying both \begin{math}p=\omega(n^{-1})\end{math} and \begin{math}p= o(n^{-1/r})\end{math}, they showed, among other results, that with probability tending to one as \begin{math}n\rightarrow \infty\end{math} either only a few additional vertices are infected or almost every vertex becomes infected. In addition they determined, depending on the number of initially infected vertices, 
the probability of both of these events up to an additive term tending to zero as \begin{math}n\rightarrow \infty\end{math}.

The main contributions of this paper are threefold. First we strengthen this result by showing exponential tail bounds. Second we introduce a martingale in order to determine the number of infected vertices during the early stages of the process. 
Finally in the supercritical regime we show that the subgraph spanned by the vertices with \begin{math}r-1\end{math} infected neighbours grows large enough to contain a giant component. The infection of just one vertex in this giant component leads to every vertex in the component becoming infected and we show that this in fact happens.

\paragraph{Main Results.}
Throughout the paper we assume that \begin{math}r\geq 2\end{math} and that both \begin{math}p=\omega(n^{-1})\end{math} and \begin{math}p=o ( n^{-1/r})\end{math} hold.
Set 
\begin{equation*}
t_0:=\left(\frac{r!}{np^r}\right)^{1/(r-1)}.
\end{equation*}
Let \begin{math}\hat{\pi}(t)=\mathbb{P}[\mathrm{Bin(t,p)}\geq r]\end{math} and define
\begin{equation*}
a_c:=-\min_{t\leq t_0} \frac{n\hat{\pi}(t)-t}{1-\hat{\pi}(t)}.
\end{equation*}
In addition denote by \begin{math}t_c\end{math} the smallest value \begin{math}t\end{math} where this minimum is reached. Similarly to \cite{MR3025687} it can be shown that 
\begin{equation*}
t_c=(1+o(1))((r-1)!/(np^r))^{1/(r-1)} \quad \mbox{and} \quad
a_c=(1+o(1))(1-1/r)t_c.
\end{equation*}

\begin{theorem}\label{mainsub}

Let \begin{math}\omega_0\end{math} be any function satisfying the conditions \begin{math}\omega_0=\omega(\sqrt{a_c})\end{math} and \begin{math}\omega_0\leq a_c-r\end{math}. \linebreak[4] If  \begin{math}|\infec(0)|=a_c-\omega_0\end{math}, then with probability at least
\begin{equation*}	
1-\exp\left(-\frac{\omega_0^2}{10 t_0}\right)
\end{equation*}
	we have \begin{math}|\infec_f|< t_c\end{math}.
\end{theorem}

\begin{theorem}\label{mainsup}
	
	Let \begin{math}\omega_0\end{math} be any function satisfying the conditions \begin{math}\omega_0=\omega(\sqrt{a_c})\end{math} and \begin{math}\omega_0\leq t_0-a_c\end{math}. \linebreak[4] If \begin{math}|\infec(0)|=a_c+\omega_0\end{math}, then with probability at least
\begin{equation*}
	1-\exp\left(-\frac{\omega_0^2}{10 t_0}\right)-\exp\left(-\frac{a_c+\omega_0}{4}\right)
\end{equation*}
	we have \begin{math}|\infec_f|= (1+o(1))n\end{math}.
\end{theorem}

\paragraph{Proof Technique.}
When the number of infected vertices is small (at most \begin{math}t_0\end{math}), we introduce a martingale to show that the number of infected vertices is concentrated around its expectation with \emph{exponentially} high probability. The martingale resembles the one introduced in \cite{MR3025687}, however the maximal one step difference in our martingale is significantly smaller and thus provides a tighter concentration bound (Lemma \ref{conc}). 

In the subcritical regime, the expected number of infected vertices is less than \begin{math}t_c<t_0\end{math} and therefore the martingale argument alone implies the result (Section \ref{subcritical}).

In the supercritical regime, this is not enough as the number of infected vertices will reach \begin{math}t_0\end{math} with exponentially high probability. In fact, at least \begin{math}t_0+a_c\end{math} vertices become infected (Lemma \ref{early}). Now take a subset of the infected vertices with size \begin{math}t_0\end{math} and consider the vertices with at least \begin{math}r-1\end{math} neighbours in this set. The size of this set is roughly \begin{math}rp^{-1}\end{math} (Lemma \ref{almostinfected}) and the subgraph spanned by these vertices is also a binomial random graph, \begin{math}G(rp^{-1},p)\end{math}. Since the seminal work of Erd\H{o}s and R\'enyi \cite{MR0125031}, it is known that this graph has with probability \begin{math}1+o(1)\end{math} a linear sized giant component. More recently, Bollob\'as and Riordan \cite{arXiv:1403.6558} showed that this happens with exponentially high probability (Theorem \ref{lingiant}). Should any vertex in the giant component have an additional infected neighbour, then every vertex in the giant will become infected eventually. We show that this happens with exponentially high probability. 

Thus we have \begin{math}\Omega(p^{-1})\end{math} infected vertices. After this, the process ends in two steps and this can be shown  by two simple applications of the Chernoff bound (Lemmas \ref{Chernoff1} and \ref {Chernoff2}).

\section{Preliminaries}

We will use the following form of the Chernoff bound.

\begin{theorem}\cite{MR2283885}\label{chernoff}
Let \begin{math}X\sim\mathrm{Bin}(n,p)\end{math}, i.e.\ a binomial random variable with parameters \begin{math}n\end{math} and \begin{math}p\end{math}. Then for any \begin{math}\lambda>0\end{math} 
\begin{equation*}
\mathbb{P}[X-\mathbb{E}(X)\leq -\lambda]\leq \exp\left(-\frac{\lambda^2}{2\mathbb{E}(X)}\right) \quad \mbox{and} \quad
\mathbb{P}[X-\mathbb{E}(X)\geq \lambda]\leq \exp\left(-\frac{\lambda^2}{2(\mathbb{E}(X)+\lambda/3)}\right).
\end{equation*}
\end{theorem}

Let \begin{math}M_0,M_1,\ldots,M_i \end{math} be a sequence of random variables and denote by \begin{math}\filter{i}\end{math} the filter generated by \begin{math}M_0,\ldots,M_i \end{math}. We say \begin{math}M_0,\ldots,M_k \end{math} forms a martingale if for every \begin{math}0\leq i \leq k\end{math} we have \begin{math}\mathbb{E}(|M_i|)<\infty \end{math} and for every \begin{math}1 \leq i \leq k \end{math}
\begin{equation*}
\mathbb{E}[M_{i}|\filter{i-1}]=M_{i-1}.
\end{equation*}
The following concentration bound on martingales due to Chung and Lu \cite{MR2283885} will prove to be vital.
\begin{theorem}\label{marconc}\cite{MR2283885}
For \begin{math}m_0\in \reals\end{math} let \begin{math}M_0=m_0,M_1,\ldots,M_k\end{math} be a martingale whose conditional variance and differences satisfy the following: for each \begin{math}1\leq i \leq k\end{math}, \begin{itemize}
\item \begin{math}\mathrm{Var}[M_i|M_{i-1},\ldots, M_0]\leq \sigma_i^2\end{math};
\item \begin{math}|M_i-M_{i-1}|\leq m\end{math} for some positive \begin{math}m\end{math}. 
\end{itemize}
Then for any \begin{math}\lambda>0\end{math}, we have
\begin{equation*}
\mathbb{P}[M_k - M_0 \geq \lambda]\leq \exp\left(-\frac{\lambda^2}{2\left(\sum_{i=1}^k \sigma_i^2+m\lambda/3 \right)}\right).
\end{equation*}
\end{theorem}

We will also need the following Theorem on the appearance of a giant component in \begin{math}G(n,p)\end{math} by Bollob\'as and Riordan \cite{arXiv:1403.6558}.

\begin{theorem}\label{lingiant}\cite{arXiv:1403.6558}
Let \begin{math}c>1\end{math} be a constant independent of \begin{math}n\end{math} and let \begin{math}\varepsilon>0\end{math} independent of \begin{math}n\end{math}. Then with probability \begin{math}1-\exp(-\Omega(n))\end{math} the binomial random graph \begin{math}G(n,c/n)\end{math} has  a component of size at least \begin{math}(1-\varepsilon)\rho n\end{math}, where \begin{math}\rho\in(0,1)\end{math} is the unique positive solution of \begin{math}1-\rho=\exp(-c\rho)\end{math}.
\end{theorem}

\section{Setup: Martingale}

In order to analyse the bootstrap percolation on \begin{math}G(n,p)\end{math} we will use the following reformulation due to Scalia-Tomba \cite{MR798872} as in \cite{MR3025687}. Roughly speaking they examine the infected vertices one by one and determine the vertices which have at least \begin{math}r\end{math} neighbours in the set of previously examined vertices. The set of examined vertices until step \begin{math}t\end{math} is denoted by \begin{math}\checked(t)\end{math} and the set of infected vertices by \begin{math}\infec(t)\end{math}.
Formally let \begin{math}\infec(0)\end{math} be the set of initially infected vertices of size \begin{math}a\end{math} and without the loss of generality we may assume that \begin{math}\infec(0)=\{1,...,a\}\end{math}. Set \begin{math}\checked(0)=\emptyset\end{math}. For each step \begin{math}t\in\naturals\end{math}, if \begin{math}\infec(t-1)\backslash \checked(t-1)\neq\emptyset\end{math}, then let \begin{math}U_t=\{u_t\}\end{math}, where \begin{math}u_t\end{math} is a vertex in \begin{math}\infec(t-1)\backslash \checked(t-1)\end{math} selected according to an arbitrary rule, otherwise set \begin{math}U_t=\emptyset\end{math}. Set \begin{math}\checked(t):=\checked(t-1)\cup U_t\end{math}. Now for \begin{math}t\geq 0\end{math} and each \begin{math}i\in[n-a]:=\{1,\ldots, n-a\}\end{math} let \begin{math}X(t,i)\end{math} be the indicator random variable for the event that the vertex \begin{math}a+i\end{math} has at least \begin{math}r\end{math} neighbours in \begin{math}\checked(t)\end{math} and set 
\begin{equation*}
\infec(t):=\infec(0)\cup\{a+i:X(t,i)=1, i\in [n-a]\}.
\end{equation*}
The process stops when \begin{math}t=n\end{math}. 

Clearly \begin{math}\checked(t)\subset \infec(t)\end{math}. Let \begin{math}T\end{math} denote the smallest value of \begin{math}t\end{math} such that \begin{math}\infec(t)=\checked(t)\end{math}. Note that \begin{math}t\leq T\end{math} implies that \begin{math}|\checked(t)|=t\end{math} and thus \begin{math}T\end{math} is also the smallest \begin{math}t\end{math} such that \begin{math}|\infec(t)|=t\end{math}. Since \begin{math}|\infec(t)|\leq n\end{math} for every natural number \begin{math}0\leq t \leq n\end{math} we have that \begin{math}T\leq n\end{math}. Note further that \begin{math}\infec(T)=\infec_f\end{math}. 

In order to have a better control on the maximal number of vertices which can become infected in a single step, we refine the process by dividing every step into rounds, in such a way that in each round \emph{exactly one vertex} \begin{math}v\in [n]\backslash \infec(0)\end{math} is examined (regardless whether it was examined in earlier rounds or not). Thus each step \begin{math}1\leq t\leq n\end{math} consists of \begin{math}n-a\end{math} rounds and round \begin{math}i\end{math} of step \begin{math}t\end{math} is denoted by \begin{math}(t,i)\end{math}. We denote the step following \begin{math}(t,i)\end{math} by \begin{math}(t,i)+1\end{math} and the preceding step by \begin{math}(t,i)-1\end{math}. Also the ordering of the rounds is given by the lexicographical order i.e. \begin{math}(\tau,\iota)<(t,i)\end{math} if either \begin{math}\tau<t\end{math} or \begin{math}\tau=t\end{math} and \begin{math}\iota<i\end{math}.

In round \begin{math}i\end{math} of step \begin{math}t\end{math} we examine if vertex \begin{math}a+i\end{math} has at least \begin{math}r\end{math} neighbours in \begin{math}\checked(t)\end{math} and if it has we add it to the set of infected vertices. Formally for \begin{math}(t,i)\geq (1,1)\end{math}
\begin{align*}
\infec((t,i)+1):=\infec(0)&\cup \{a+j: j\leq i, X(t,j)=1\}\cup \{a+j: j> i, X(t-1,j)=1\}.
\end{align*}
Clearly we have \begin{math}\infec(t)=\infec(t,n-a)\end{math}. For consistency define \begin{math} \infec(0,n-a):=\infec(0) \end{math}.

Define a function \begin{math}\pi:\naturals\rightarrow [0,1]\end{math} by
\begin{equation*}
\pi(t):=\left\{
\begin{array}{ll}
\mathbb{P}[\mathrm{Bin(t,p)}\geq r],  & \mbox{for } t\leq T\\ 
\mathbb{P}[\mathrm{Bin(T,p)}\geq r], & \mbox{for } t >T 
\end{array}
\right.
\end{equation*}
and note that \begin{math}\pi(t)\end{math} is a random variable.

For \begin{math}(t,i)\geq (0,n-a) \end{math}, define the random variable 
\begin{equation}\label{martingaledef}
M(t,i):= \sum_{j=1}^i \frac{X(t,j)-\pi(t)}{1-\pi(t)}+\sum_{j=i+1}^{n-a}\frac{X(t-1,j)-\pi(t-1)}{1-\pi(t-1)}.
\end{equation}
We will denote by \begin{math}\filter{t,i}\end{math} the filter generated by \begin{math}M(0,n-a),\ldots,M(t,i)\end{math}.

\begin{lemma}
The sequence of random variables \begin{math}M(0,n-a),\ldots, M(n,n-a)\end{math} forms a martingale.
\end{lemma}

\begin{proof}
Fix \begin{math}1\leq t\leq n\end{math} and \begin{math}1\leq i \leq n-a\end{math}. For every \begin{math}\tau< t\end{math}, we can express from \eqref{martingaledef} the number of infected vertices in step \begin{math}\tau\end{math}:
\begin{equation}\label{marinfec}
|\infec(\tau)|=a+\sum_{\iota=1}^{n-a}X(\tau,\iota)\stackrel{\eqref{martingaledef}}{=}a+M(\tau,n-a)(1-\pi(\tau))+(n-a)\pi(\tau).
\end{equation}

Recall that \begin{math}\hat{\pi}(t)=\mathbb{P}[\mathrm{Bin}(t,p)\geq r]\end{math}. We will denote by \begin{math}T'\end{math} the smallest value of \begin{math}t\end{math} which satisfies 
 \begin{math}a+M(t,n-a)(1-\hat{\pi}(t))+(n-a)\hat{\pi}(t)=t\end{math}.
 Because \begin{math}\pi(t)=\hat{\pi}(t)\end{math} when \begin{math}t\leq T\end{math}, we have \begin{math}T=T'\end{math}. Given the 
filter \begin{math}\filter{(t,i)-1}\end{math} one can establish if \begin{math}a+M(\tau,n-a)(1-\hat{\pi}(\tau))+(n-a)\hat{\pi}(\tau)=\tau\end{math} for some \begin{math}\tau<t\end{math}. Therefore, it can be determined whether the event \begin{math}T'< t\end{math} or \begin{math}t\geq T'\end{math} holds. In particular, if \begin{math}T'< t\end{math}, then the exact value of \begin{math}T'\end{math} can be determined. 

For each \begin{math}\tau\leq t\end{math}, since \begin{math}\pi(\tau)\end{math} depends only on the value of \begin{math}T=T'\end{math}, we can also determine the value of \begin{math}\pi(\tau)\end{math}, i.e.\  \begin{math}\mathbb{E}[\pi(\tau)|\filter{(t,i)-1}]=\pi(\tau)\end{math} for \begin{math}\tau\leq t\end{math}. 

Note that \begin{math}X(0,i)=0\end{math} for every \begin{math}1 \leq i \leq n-a\end{math} and that for every \begin{math}(1,1)\leq (\tau,\iota)<(t,i)\end{math} we can easily compute from \eqref{martingaledef}
\begin{equation}\label{onestepdiff}
M(\tau,\iota)-M((\tau,\iota)-1)=\frac{X(\tau,\iota)-\pi(\tau)}{1-\pi(\tau)}-\frac{X(\tau-1,\iota)-\pi(\tau-1)}{1-\pi(\tau-1)}.
\end{equation}
Therefore, based on the 
filter \begin{math}\filter{(t,i)-1}\end{math}, the value of \begin{math}X(\tau,\iota)\end{math} can be determined for every \begin{math}(\tau,\iota)<(t,i)\end{math}.

Next we shall show that
\begin{align}
\mathbb{E}\left[\left. \frac{X(t,i)-\pi(t)}{1-\pi(t)}\right| \filter{(t,i)-1}\right]&=\frac{\mathbb{E}[X(t,i)| \filter{(t,i
)-1}]-\pi(t)}{1-\pi(t)}=\frac{X(t-1,i)-\pi(t-1)}{1-\pi(t-1)}.\label{expectation}
\end{align}
To this end, observe that if \begin{math}X(t-1,i)=1\end{math}, then we have \begin{math}X(t,i)=1\end{math} with probability 1 and in this case both sides of equation \eqref{expectation} equal 1. 

Now assume that \begin{math}X(t-1,i)=0\end{math}. When \begin{math}t>T\end{math}, we have \begin{math}X(t,i)=X(t,i-1)=0\end{math} with probability 1 and by the definition of \begin{math}\pi(t)\end{math} we have \begin{math}\pi(t)=\pi(t-1)=\hat{\pi}(T)\end{math}. Evaluating both sides of equation \eqref{expectation} gives us  \begin{math}-\hat{\pi}(T)/(1-\hat{\pi}(T))\end{math}. 
When \begin{math}t\leq T\end{math} and \begin{math}X((t,i)-1)=0\end{math}, we have \begin{math}\pi(t)=\hat{\pi}(t)\end{math} and thus
\begin{align*}
\mathbb{P}\left[X(t,i)=0| \filter{(t,i)-1}\right]&=\frac{1-\hat{\pi}(t)}{1-\hat{\pi}(t-1)}
=1-\frac{\hat{\pi}(t)-\hat{\pi}(t-1)}{1-\hat{\pi}(t-1)}.
\end{align*}
Therefore in this case
\begin{align*}
\mathbb{E}\left[\left.\frac{X(t,i)-\pi(t)}{1-\pi(t)}\right| \filter{(t,i)-1}\right]&=-\frac{\hat{\pi}(t)}{1-\hat{\pi}(t)} \frac{1-\hat{\pi}(t)}{1-\hat{\pi}(t-1)}+1\cdot\frac{\hat{\pi}(t)-\hat{\pi}(t-1)}{1-\hat{\pi}(t-1)}=-\frac{\hat{\pi}(t-1)}{1-\hat{\pi}(t-1)}
\end{align*}
and thus \eqref{expectation} holds.
According to \eqref{onestepdiff} we have that
\begin{align*}
\mathbb{E}&[M(t,i)-M((t,i)-1)|\filter{(t,i)-1}]\\
&=\mathbb{E}\left[\left.\frac{X(t,i)-\pi(t)}{1-\pi(t)}\right| \filter{(t,i)-1}\right]-\frac{X(t-1)-\pi(t-1)}{1-\pi(t-1)}\stackrel{\eqref{expectation}}{=}0.
\end{align*}
\end{proof}

\begin{lemma}\label{conc}
Let \begin{math}t\in \{0,\ldots,n\}\end{math} and \begin{math}\lambda\in \reals^+\end{math} be given.
\begin{equation*}
\mathbb{P}\left[\bigwedge_{(0,n-a)\leq(\tau,i)\leq (t,n-a)} M(\tau,i)>-\lambda \right]\geq 1-\exp\left(-\frac{\lambda^2(1-\hat{\pi}(t))^3}{2(n\hat{\pi}(t)+\lambda/3)}\right)
\end{equation*}
and
\begin{equation*}
\mathbb{P}\left[\bigwedge_{(0,n-a)\leq(\tau,i)\leq (t,n-a)} M(\tau,i)<\lambda \right]\geq 1-\exp\left(-\frac{\lambda^2(1-\hat{\pi}(t))^3}{2(n\hat{\pi}(t)+\lambda/3)}\right).
\end{equation*}
\end{lemma}

\begin{proof}
We will only show the bound on the probability that \begin{math}M(\tau,i)< \lambda\end{math} for each \begin{math}(\tau,i)\leq (t,n-a)\end{math}. The other case follows simply from the fact that if the random variables \begin{math}M(0,n-a),\ldots,M(t,n-a)\end{math} form a martingale, then \begin{math}-M(0,n-a),\ldots,-M(t,n-a)\end{math} is also a martingale and they both have the same conditional variance and maximal difference.
In order to show that the bounds hold for each round, we introduce the following martingale:
\begin{equation*}
\hat{M}(\tau,i)=\left\{
\begin{array}{ll}
M(\tau,i) &\mbox{if } \hat{M}((\tau,i)-1)< \lambda \\
\hat{M}((\tau,i)-1) & \mbox{otherwise}.
\end{array}
\right. 
\end{equation*}
Similarly to \begin{math}M(t,i)\end{math} we denote the filter generated by \begin{math} \hat{M}(0,n-a),\ldots,\hat{M}(t,i) \end{math} with \begin{math} \hat{\mathcal{F}}(t,i) \end{math}.
Note that if there exists a round \begin{math}(\tau,i)\end{math} such that \begin{math}M(\tau,i)\geq \lambda\end{math}, then we have \begin{math}\hat{M}(\tau',i')\geq \lambda\end{math} for every \begin{math}(\tau',i')\geq (\tau,i)\end{math}. Therefore \begin{math}\hat{M}(t,n-a)< \lambda\end{math} implies that for every \begin{math}(\tau,i)\leq (t,n-a)\end{math} we have \begin{math}M(\tau,i)< \lambda\end{math}.

By \eqref{onestepdiff} and since \begin{math}\pi(\tau)\leq \hat{\pi}(\tau)\end{math} with probability 1, we have 
\begin{equation*}
|\hat{M}(\tau,\iota)-\hat{M}((\tau,\iota)-1)|\leq\max\left\{\frac{\hat{\pi}(\tau)}{1-\hat{\pi}(\tau)}-\frac{\hat{\pi}(\tau-1)}{1-\hat{\pi}(\tau-1)},1+\frac{\hat{\pi}(\tau-1)}{1-\hat{\pi}(\tau-1)}\right\}
=\frac{1}{1-\hat{\pi}(\tau-1)}.
\end{equation*}
Note that \begin{math}M(0,n-a)=0\end{math}. 
Since \begin{math}\mathrm{Var}[\hat{M}(\tau,i)|\hat{\mathcal{F}}((\tau,i)-1)]=0\end{math} if \begin{math}\hat{M}((\tau,\iota)-1)\geq \lambda\end{math} and 
\begin{equation*}
\mathrm{Var}[\hat{M}(\tau,i)|\hat{\mathcal{F}}((\tau,i)-1)]= 
\mathrm{Var}[M(\tau,i)|\filter{(\tau,i)-1}]
\end{equation*}
otherwise, Theorem \ref{marconc} implies that
\begin{equation*}
\mathbb{P}[\hat{M}(t,n-a)\geq \lambda]\leq \exp\left(-\frac{\lambda^2}{2(S+\lambda/(3(1-\hat{\pi}(\tau-1))))}\right)
\end{equation*}
where 
\begin{equation}\label{sumvar}
S \leq \sum_{\tau=1}^{t}\sum_{i=1}^{n-a}\mathrm{Var}[M(\tau,i)|\filter{(\tau,i)-1}].
\end{equation}
Note that 
\begin{align}
\mathrm{Var}[M(\tau,i)|\filter{(\tau,i)-1}]&=\mathrm{Var}\left[\left.\frac{X(\tau,i)}{1-\pi(t)}\right| \filter{(\tau,i)-1}\right]\nonumber\\
&\leq\frac{1}{(1-\hat{\pi}(t))^2}\mathrm{Var}[X(\tau,i)| \filter{(\tau,i)-1}].
\end{align}

Recall that \begin{math}X(\tau-1,i)=1\end{math} implies \begin{math}X(\tau,i)=1\end{math} and that \begin{math}\tau>T\end{math} implies \begin{math}X(\tau,i)=X(\tau-1,i)\end{math}. 
In both of these cases we have
\begin{equation}
\mathrm{Var}[X(\tau,i)| \filter{(\tau,i)-1}]=0\stackrel{\hat{\pi}(t)\geq \hat{\pi}(t-1)}{\leq} \frac{\hat{\pi}(t)-\hat{\pi}(t-1)}{1-\hat{\pi}(t)}.
\end{equation}
Now assume \begin{math}\tau\leq T\end{math} and \begin{math}X(\tau-1,i)=0\end{math}. Since \begin{math}X(t,i)\end{math} is an indicator random variable, we have
\begin{align}
\mathrm{Var}[X(\tau,i)| \filter{(\tau,i)-1}]&\leq \mathbb{E}[X(\tau,i)| \filter{(\tau,i)-1}]=\frac{\hat{\pi}(\tau)-\hat{\pi}(\tau-1)}{1-\hat{\pi}(\tau-1)}\leq\frac{\hat{\pi}(\tau)-\hat{\pi}(\tau-1)}{1-\hat{\pi}(\tau)}.\label{exvar}
\end{align}

Putting \eqref{sumvar}-\eqref{exvar} together, we obtain
\begin{equation*}
S\leq \sum_{\tau=1}^t\sum_{i=1}^{n-a}\frac{\hat{\pi}(\tau)-\hat{\pi}(\tau-1)}{(1-\hat{\pi}(\tau))^3}\leq \sum_{\tau=1}^t \frac{n(\hat{\pi}(\tau)-\hat{\pi}(\tau-1))}{(1-\hat{\pi}(t))^3}\leq \frac{n\hat{\pi}(t)}{(1-\hat{\pi}(t))^3}.
\end{equation*}

\end{proof}

The previous lemma allows us to analyse the process in the first \begin{math}t_0\end{math} steps. This will be used in the proofs of Theorems \ref{mainsub} and \ref{mainsup}.

\section{Proof of Theorem \ref{mainsub}}\label{subcritical}
We want to investigate the number of infected vertices at time \begin{math}t_c\end{math}. By the definition of \begin{math}a_c\end{math} and \begin{math}t_c\end{math}, we have
\begin{equation}\label{critical}
a_c=-\min_{t\leq t_0}\frac{n\hat{\pi}(t)-t}{1-\hat{\pi}(t)}=\frac{t_c-n\hat{\pi}(t_c)}{1-\hat{\pi}(t_c)}.
\end{equation}
By the definition of \begin{math}M(t,i)\end{math}, we have
\begin{align*}
|\infec(t_c)|&\stackrel{\eqref{marinfec}}{=}a+(1-\pi(t_c))M(t_c,n-a)+(n-a)\pi(t_c).\\
\end{align*}
Since \begin{math}\pi(t)\leq \hat{\pi}(t)\end{math} and \begin{math}a=a_c-\omega_0\end{math}, we obtain
\begin{align}
|\infec(t_c)|&\leq a+M(t_c,n-a)+(n-a)\hat{\pi}(t_c)\nonumber\\
&=(a_c-\omega_0)(1-\hat{\pi}(t_c))+n\hat{\pi}(t_c)+M(t_c,n-a)\nonumber\\
&\stackrel{\eqref{critical}}{=}t_c-n\hat{\pi}(t_c)+n\hat{\pi}(t_c)-\omega_0(1-\hat{\pi}(t_c))+M(t_c,n-a)\nonumber\\
&\stackrel{\hat{\pi}(t_c)\leq \hat{\pi}(t_0)}{\leq} t_c-\omega_0(1-\hat{\pi}(t_0))+M(t_c,n-a).\label{numinfver}
\end{align}
Using \begin{math}np=\omega(1)\end{math} and \begin{math}t_0=(r!/(np^r))^{1/(r-1)}\end{math}, we have
\begin{equation}\label{t0small}
t_0p=O\left(\left(\frac{1}{np^r}\right)^{1/(r-1)}p\right)=O\left(\left(\frac{1}{np}\right)^{1/(r-1)}\right)=o(1).
\end{equation}
Furthermore,
\begin{align}
\hat{\pi}(t_0)&=\mathbb{P}[\mathrm{Bin}(t_0,p)\geq r]=\sum_{j=r}^{t_0}\binom{t_0}{j}p^{j}(1-p)^{t_0-j}\stackrel{\eqref{t0small}}{=}(1+o(1))\frac{t_0^r p^r}{r!}\nonumber\\
&=(1+o(1))\frac{t_0^{r-1}p^r}{r!}t_0=(1+o(1))\frac{t_0}{n}\stackrel{np=\omega(1)}{=}o\left(t_0p\right)\stackrel{\eqref{t0small}}{=}o(1).\label{probt0}
\end{align}
Applying Lemma \ref{conc} with \begin{math}\lambda=\omega_0/2\end{math}, we have that with probability at least
\begin{equation*}
1-\exp\left(-\frac{\omega_0^2}{8((1+o(1))n\hat{\pi}(t_c)+\omega_0/6)}\right)\geq 1-\exp\left(-\frac{\omega_0^2}{10 t_0}\right)
\end{equation*}
\begin{math}M(t_c,n-a)< \omega_0/2\end{math}.
This together with \eqref{numinfver} implies 
\begin{equation*}
|\infec(t_c)|\leq t_c-(1+o(1))\omega_0+\omega_0/2<t_c.
\end{equation*}
 Therefore \begin{math}T<t_c\end{math} and thus \begin{math}|\infec_f|=T< t_c\end{math}.

\section{Proof of Theorem \ref{mainsup}}

Before proving Theorem \ref{mainsup} we begin with an observation on \begin{math}A(t_0)\end{math}, the set of infected vertices after the first \begin{math}t_0\end{math} steps.

\begin{lemma}\label{early}
	Let \begin{math}\omega_0\end{math} be any function satisfying the conditions \begin{math}\omega_0=\omega(\sqrt{a_c})\end{math} and \begin{math}\omega_0\leq t_0-a_c\end{math}. \linebreak[4] If \begin{math}|\infec(0)|=a_c+\omega_0\end{math}, then with probability at least
\begin{equation*}
1-\exp\left(-\frac{\omega_0^2}{9.5 t_0}\right)
\end{equation*}
we have \begin{math}T>t_0\end{math} and \begin{math}|\infec(t_0)|\geq t_0+(1+o(1))a_c+\omega_0/2\end{math}.
\end{lemma}

\begin{proof}
By the definition of \begin{math}a_c\end{math}, for every \begin{math}t\leq t_0\end{math} we have
\begin{equation}\label{acmin}
a_c\geq \frac{t-n\hat{\pi}(t)}{1-\hat{\pi}(t)}.
\end{equation}

Assume that \begin{math}M(t,i)> -\omega_0/2\end{math} for every \begin{math}t\leq t_0\end{math} and \begin{math}1\leq i \leq n-a\end{math}. First we will show by induction that if \begin{math}M(t,i)\end{math} satisfies this lower bound, then \begin{math}T>t_0\end{math}. Clearly \begin{math}T>0\end{math}. Now assume that for some \begin{math}t\leq t_0-1\end{math} we have that \begin{math}T>t-1\end{math}. Therefore \begin{math}\pi(t)=\hat{\pi}(t)\end{math}. For \begin{math}t\leq t_0\end{math} we have \begin{math}\hat{\pi}(t)\leq \hat{\pi}(t_0)\stackrel{\eqref{probt0}}{=}o(1)\end{math} and thus

\begin{align*}
|\infec(t)|&\stackrel{\eqref{marinfec}}{=}a+(1-\pi(t))M(t,n-a)+(n-a)\pi(t)\\
&\stackrel{M(t,i)> -\omega_0/2}{\geq} (1-\hat{\pi}(t))(a_c+\omega_0)+n\hat{\pi}(t)-(1+o(1))\omega_0/2\\
&\stackrel{\eqref{acmin}}{\geq} t+(1+o(1))\omega_0/2.
\end{align*}
Therefore \begin{math}|\infec(t)|>t\end{math} which together with \begin{math}T>t-1\end{math} implies \begin{math}T>t\end{math}.

Also note that
\begin{align}
|\infec(t_0)|&\stackrel{\eqref{marinfec}}{=}a+M(t_0,n-a)(1-\pi(t_0))+(n-a)\hat{\pi}(t_0)\nonumber\\
&\stackrel{\eqref{probt0}}{\geq}(1+o(1))a_c+(1+o(1))t_0+(1+o(1))\omega_0/2.\label{infect0}
\end{align}

Let \begin{math}t_1:=((r-1)!/np^r)^{1/(r-1)}\end{math}. Then \begin{math}t_1\leq t_0\end{math} and so
\begin{equation}\label{mincond}
a_c\geq \frac{t_1-n\hat{\pi}(t_1)}{1-\hat{\pi}(t_1)}.
\end{equation}
Also
\begin{align}
\hat{\pi}(t_1)&=\mathbb{P}[\mathrm{Bin}[t_1,p)\geq r]=\sum_{j=r}^{t_0}\binom{t_0}{j}p^{j}(1-p)^{t_1-j}\stackrel{\eqref{t0small}}{=}(1+o(1))\frac{t_1^r p^r}{r!}\nonumber\\
&=(1+o(1))\frac{t_1^{r-1}p^r}{r!}t_1=(1+o(1))\frac{t_1}{rn}.
\end{align}
From this and \eqref{mincond}
\begin{equation}
a_c\geq \frac{t_1-(1+o(1))t_1/r}{1-\hat{\pi}(t_1)}=(1+o(1))\left(1-\frac{1}{r}\right)t_1=\Omega(t_0),
\end{equation}
which together with \eqref{infect0} and \begin{math}\omega_0\leq t_0\end{math} implies
\begin{equation*}
|\infec(t_0)|=t_0+(1+o(1))a_c+\omega_0/2.
\end{equation*}
Lemma \ref{conc} with \begin{math}\lambda=\omega_0/2\end{math} implies the result.

\end{proof}

We will need to establish the size of the giant component in the set of vertices which have at least \begin{math}r-1\end{math} neighbours in \begin{math}\checked(t_0)\end{math}. For this we first need to determine the number of vertices which have at least \begin{math}r-1\end{math} neighbours in \begin{math}\checked(t_0)\end{math}.

\begin{lemma}\label{almostinfected}
Let \begin{math}A\subset [n]\end{math} with \begin{math}|A|=o(n)\end{math}. Conditional on \begin{math}T\geq t_0\end{math}, \begin{math}\infec(t_0)\end{math} and \begin{math}\checked(t_0)\end{math}, with probability \begin{math}1-\exp(-\Omega(p^{-1}))\end{math} we have that the number of vertices in \begin{math}[n]\backslash (\checked(t_0)\cup A)\end{math} with at least \begin{math}r-1\end{math} neighbours in \begin{math}\checked(t_0)\end{math} is at least \begin{math}3rp^{-1}/4\end{math}.
\end{lemma}

\begin{proof}
Let \begin{math}X_v\end{math} be the indicator random variable that a vertex \begin{math}v\in [n] \backslash (\checked(t_0)\cup A)\end{math} has at least \begin{math}r-1\end{math} neighbours in \begin{math}\checked(t_0)\end{math} and set \begin{math}X:=\sum_{v\in [n] \backslash (\checked(t_0)\cup A)}X_v\end{math}.
Clearly
\begin{equation*}
\mathbb{P}[X_v=1|v\in \infec(t_0)]=1.
\end{equation*}
Note that if \begin{math}X_v=1\end{math} and \begin{math}v\not\in\infec(t_0)\end{math}, then \begin{math}v\end{math} has exactly \begin{math}r-1\end{math} neighbours in \begin{math}\checked(t_0)\end{math} and thus
\begin{align*}
\mathbb{P}[X_v=1|v\not\in\infec(t_0)]&\geq \mathbb{P}[X_v=1,v\not\in\infec(t_0)]=\binom{t_0}{r-1}p^{r-1}(1-p)^{t_0-r+1}\\
&\stackrel{\eqref{t0small}}{=}(1+o(1)) \frac{t_0^{r-1}}{(r-1)!}p^{r-1}=(1+o(1))\frac{r}{np}.
\end{align*}

Since the set of random variables \begin{math}\{X_v|v\in [n]\backslash (\checked(t_0) \cup A)\}\end{math} are mutually independent, \begin{math}X\end{math} stochastically dominates the binomial random variable \begin{math}\hat{X}=\mathrm{Bin}(n-t_0-|A|,(1+o(1))r/np)\end{math}.
Because \begin{math}t_0=o(n)\end{math} and \begin{math}|A|=o(n)\end{math}, we have
\begin{equation*}
\mathbb{E}[\hat{X}]=(1+o(1))rp^{-1}
\end{equation*}
and Theorem \ref{chernoff} implies
\begin{align*}
\mathbb{P}[X-\mathbb{E}(\hat{X})&\leq -(1+o(1))rp^{-1}/4]\leq \exp\left(-(1+o(1))\frac{r^2p^{-2}}{32rp^{-1}}\right)\leq \exp\left(-\frac{p^{-1}}{32}\right).
\end{align*}
\end{proof}

In the following two lemmas we look at the number of vertices with at least \begin{math}r\end{math} neighbours in a set of order \begin{math}p^{-1}\end{math} and set of order \begin{math}n\end{math}. Estimating the probability that a vertex has at least \begin{math}r\end{math} neighbours in such sets differs significantly  and are discussed separately.

\begin{lemma}\label{Chernoff1}
Let \begin{math}U,W\subset [n]\end{math} in \begin{math}G(n,p)\end{math} satisfy \begin{math}|U|= p^{-1}/2\end{math} and \begin{math}|W|=o(n)\end{math}. With probability at least \begin{math}1-\exp(-\Omega(n))\end{math} the number of vertices in \begin{math}[n]\backslash(U\cup W)\end{math} with at least \begin{math}r\end{math} neighbours in \begin{math}U\end{math} is at least \begin{math}n/(2^{r}r!\sqrt{\mathrm{e}})\end{math}.
\end{lemma}

\begin{proof}
Let \begin{math}Y_v\end{math} be the indicator random variable that a vertex \begin{math}v\in [n] \backslash (U \cup W)\end{math} has at least \begin{math}r\end{math} neighbours in \begin{math}U\end{math} and set \begin{math}Y=\sum_{v\in [n] \backslash (U \cup W)}Y_v\end{math}.
We have that 
\begin{align*}
\mathbb{P}[Y_v=1]
&=1-\sum_{j=0}^{r-1}\binom{p^{-1}/2}{j}p^{j}(1-p)^{p^{-1}/2-j}\\
&=1-(1-p)^{p^{-1}/2}\sum_{j=0}^{r-1}\binom{p^{-1}/2}{j}\left(\frac{p}{1-p}\right)^{j}\\
&=1-(1+o(1))\mathrm{e}^{-1/2} \sum_{j=0}^{r-1}\frac{(2p)^{-j}}{j!}p^{j}\\
&=1-(1+o(1))\mathrm{e}^{-1/2} \sum_{j=0}^{r-1}\frac{1}{j!2^{j}}.
\end{align*}

Clearly \begin{math}1\geq \mathbb{P}[Y_v=1] > 1/(2^{r}r!\sqrt{\mathrm{e}})\end{math}. Set \begin{math}\eta:=\mathbb{P}[Y_v=1]-1/(2^{r}r!\sqrt{\mathrm{e}})\end{math} and note that \linebreak[4] \begin{math}\eta=\Omega(1)\end{math}. 
Furthermore the set of random variables \begin{math}\{Y_v|v\in [n]\backslash (U \cup W)\}\end{math} are mutually independent and \begin{math}|[n]\backslash (U \cup W)|=(1+o(1))n\end{math}. Therefore, by Theorem \ref{chernoff} we have 
\begin{equation*}
\mathbb{P}[Y<n/(2^{r}r!\sqrt{\mathrm{e}})]\leq \exp\left(-(1+o(1))\frac{\eta^2n^2}{2n}\right)=\exp(-\Omega(n)).
\end{equation*}
\end{proof}

\begin{lemma}\label{Chernoff2}
Let \begin{math}U,W\subset [n]\end{math} in \begin{math}G(n,p)\end{math} satisfy \begin{math}|U|= n/(2^{r}r!\sqrt{\mathrm{e}})\end{math}. Then with probability \begin{math}\exp(-\Omega(p^{-1}))\end{math} all but at most \begin{math}p^{-1}\end{math} vertices in \begin{math}[n]\backslash (U \cup W)\end{math} have at least \begin{math}r\end{math} neighbours in \begin{math}U\end{math}.
\end{lemma}

\begin{proof}
Let \begin{math}B_v\end{math} be the indicator random variable that a vertex \begin{math}v\in [n] \backslash (U \cup W)\end{math} has less than \begin{math}r\end{math} neighbours in \begin{math}U\end{math} and set \begin{math}B:=\sum_{v\in [n] \backslash (U \cup W)}B_v\end{math}.
Since \begin{math}np=\omega(1)\end{math}, we have 
\begin{align*}
\mathbb{P}[B_v=1]&=\sum_{j=0}^{r-1}\binom{|U|}{j}p^{j}(1-p)^{|U|-j}\leq \exp(-|U|p) \sum_{j=0}^{r-1} \frac{(np)^{j}}{j!}\leq \exp(-|U|p) (np)^{r-1}.
\end{align*}

Since \begin{math}|[n]\backslash (U \cup W)|\leq n\end{math} and \begin{math}np=\omega(1)\end{math}, we have
\begin{equation*}
\mathbb{E}[B]\leq \exp(-|U|p) n(np)^{r-1}=\exp(-\Omega(np)) n^r p^{r-1}=o(p^{-1}).
\end{equation*}

Note that the set of random variables \begin{math}\{B_v|v\in [n]\backslash (U \cup W)\}\end{math} are mutually independent and therefore, by Theorem \ref{chernoff} we have 
\begin{equation*}
\mathbb{P}[B>p^{-1}]\leq \exp\left(-\Omega(p^{-1})\right).
\end{equation*}
\end{proof}

\begin{proof}[of Theorem \ref{mainsup}]
According to Lemma \ref{early} with probability at least \begin{math}1-\exp(-\omega_0^2/(9.5t_0))\end{math} we have that \begin{math}|\infec(t)|>t\end{math} for every \begin{math}t\leq t_0\end{math} and \begin{math}\infec(t_0)\geq t_0+(1+o(1))a_c+\omega_0/2\end{math}. Therefore the process runs for at least \begin{math}t_0\end{math} steps and there exists a set \begin{math}A\subseteq |\infec(t_0)\backslash\checked(t_0)|\end{math} of size \begin{math}a_c/2+\omega_0/2\end{math}. 

Lemma \ref{almostinfected} implies that conditional on \begin{math}\infec(t_0)\end{math} and \begin{math}\checked(t_0)\end{math} with probability at least \begin{math}1-\exp(-\Omega(p^{-1}))\end{math} there is a set of vertices in \begin{math}[n]\backslash (\checked(t_0)\cup A)\end{math} with size at least \begin{math}3rp^{-1}/4\end{math} where every vertex in the set has at least \begin{math}r-1\end{math} neighbours in \begin{math}\checked(t_0)\end{math} and select a subset \begin{math}W\end{math} of these vertices of size exactly \begin{math}3rp^{-1}/4\end{math}. Note that until this point every event depends only on edges with one end in \begin{math}\checked(t_0)\end{math}.

According to Theorem \ref{lingiant}, with probability \begin{math}1-\exp(-\Omega(p^{-1}))\end{math} there is a set \begin{math}U\subset W\end{math} such that the vertices in \begin{math}U\end{math} form a connected component and \begin{math}|U|\geq (1-\varepsilon)\rho p^{-1}\end{math} for arbitrary \begin{math}\varepsilon>0\end{math} independent of \begin{math}n\end{math}, where \begin{math}\rho\end{math} is the unique solution of \begin{math}1-\rho=\exp(-3\rho r/4)\end{math}. 
Since 
\begin{equation*}
1-\rho>\sum_{k=0}^4 \frac{(-3\rho r/4)^k}{k!}>\exp(-3\rho r/4)
\end{equation*}
when \begin{math}0<\rho \leq 1/2\end{math}, we have \begin{math}\rho>1/2\end{math} and thus we have that \begin{math}|U|\geq p^{-1}/2\end{math}. Also this event depends only on the edges with both endpoints in \begin{math}U\end{math} and thus it is independent of the previous events. 

Note that if a vertex in \begin{math}A\end{math} is connected to a vertex in \begin{math}U\end{math}, then every vertex in \begin{math}U\end{math} will become infected. 
The probability that no vertex in \begin{math}A\end{math} is connected to any vertex in \begin{math}U\end{math} is 
\begin{equation*}
(1-p)^{|U|(a_c+\omega)/2}\leq \exp(-(a_c+\omega)/4).
\end{equation*}
This event depends on edges between \begin{math}A\end{math} and \begin{math}U\end{math} and thus it is independent of the previous events.

Now take \begin{math}U'\subset U\end{math} such that \begin{math}|U'|=p^{-1}/2\end{math} and denote the set of vertices in \begin{math}[n]\backslash (W \cup A \cup \checked(t_0))\end{math} which have at least \begin{math}r\end{math} neighbours in \begin{math}U'\end{math} with \begin{math}B\end{math}.  Since \begin{math}|W \cup A \cup \checked(t_0)|=o(n)\end{math} by Lemma \ref{Chernoff1} with probability \begin{math}1-\exp(-\Omega(n))\end{math} we have that \begin{math}|B|\geq n/(2^r r!\sqrt{\mathrm{e}})\end{math}. Note that \begin{math}B\subset |\infec_f|\end{math}. This event depends only on edges between \begin{math}U'\end{math} and \begin{math}[n]\backslash (U\cup W \cup A \cup \checked(t_0))\end{math} and thus it is independent of the previous events. 

Finally let \begin{math}B'\subset B\end{math} with \begin{math}|B'|=n/(2^r r!\sqrt{\mathrm{e}})\end{math} and consider the set of vertices in \begin{math}[n]\backslash (B \cup U \cup \checked(t_0) \cup A)\end{math} which contain at least \begin{math}r\end{math} neighbours in \begin{math}B'\end{math}. Note that all of these vertices will be infected. By Lemma \ref{Chernoff2} we have with probability \begin{math}1-\exp(-\Omega(p^{-1}))\end{math} that all but at most \begin{math} p^{-1}\end{math} vertices in \begin{math}[n]\backslash (B \cup U \cup \checked(t_0) \cup A)\end{math} will become infected. Similarly as before this event depends only on edges which we haven't considered previously and thus it is independent of the previous events. Recall that \begin{math}B \cup U \cup \checked(t_0) \cup A \subset \infec_f\end{math} and thus \begin{math}|\infec_f|=(1+o(1))n\end{math}.

Since \begin{math}p^{-1}=\omega(t_0)\end{math} and \begin{math}n=\omega(p^{-1})\end{math}, 
we have that the probability that almost every vertex becomes infected is at least
\begin{equation*}
1-\exp\left(-\frac{\omega_0^2}{10t_0}\right)-\exp\left(-\frac{a_c+\omega_0}{4}\right).
\end{equation*}
\end{proof}

\bibliographystyle{alpha}
\bibliography{referencegnpexp_conf}

\end{document}